\documentclass[10pt]{amsart}
\usepackage{amssymb,latexsym,amsmath,amsthm,graphicx,epsfig,amssymb}
\usepackage{ifthen}
\usepackage{bm}
\usepackage{bbm}
\usepackage[all]{xy}
\usepackage{pgf}
\usepackage{tikz-cd}
\usepackage{color}

\newtheorem{theorem}{Theorem}[section]
\newtheorem{lemma}[theorem]{Lemma}
\newtheorem{definition}[theorem]{Definition}
\newtheorem{remark}[theorem]{Remark}
\newtheorem{corollary}[theorem]{Corollary}
\newtheorem{proposition}[theorem]{Proposition}
\newtheorem{ex}[theorem]{Example}

\def\F{\mathcal{F}}

\def\M{\mathfrak{m}}
\def\N{\mathfrak{n}}

\begin{document}
\title{\bf Recursive Betti numbers for Cohen-Macaulay $d$-partite clutters arising from posets}
\author{Davide Bolognini}
\begin{abstract}
\noindent A natural extension of bipartite graphs are $d$-partite clutters, where $d \geq 2$ is an integer. For a poset $P$, Ene, Herzog and Mohammadi introduced the $d$-partite clutter $\mathcal{C}_{P,d}$ of multichains of length $d$ in $P$, showing that it is Cohen-Macaulay. We prove that the cover ideal of $\mathcal{C}_{P,d}$ admits an $x_i$-splitting, determining a recursive formula for its Betti numbers and generalizing a result of Francisco, H\`a and Van Tuyl on the cover ideal of Cohen-Macaulay bipartite graphs. Moreover we prove a Betti splitting result for the Alexander dual of a Cohen-Macaulay simplicial complex. 
\end{abstract}
\maketitle

\section{Introduction}
Edge ideals of graphs have been extensively studied by several authors (see for instance \cite{osc}, \cite{scm}, \cite{tuyl2}, \cite{tuyl3}, \cite{van}). In recent years the interest focused on a generalization of the notion of graph, the so-called {\em clutter} (see \cite{faridi2}, \cite{morey}, \cite{morey2}, \cite{hamorvill}, \cite{tuyl2}, \cite{wood}). 

Edge ideals have been introduced by Villarreal in \cite{vill}. Let $\mathbbm{k}$ be a field and $G$ a finite simple graph on $n$ vertices. The edge ideal of $G$ is $$I(G):=(x_ix_j:\{i,j\} \text{   is an edge of   }G) \subseteq S=\mathbbm{k}[x_1,...,x_n].$$ A graph $G$ is called Cohen-Macaulay if $S/I(G)$ is Cohen-Macaulay. Similarly it is possible to define the edge ideal $I(\mathcal{C})$ of a clutter $\mathcal{C}$.


A relevant class of graphs consists of bipartite graphs. A finite simple graph $G$ is bipartite if there exists a partition of its vertex set $V=V_1 \mathcal{t} V_2$, such that every edge of $G$ is of the form $\{i,j\}$, with $i \in V_1$ and $j \in V_2$. For every $d \geq 2$, $d$-partite clutters are a natural extension of bipartite graphs. They have been introduced by Stanley in \cite{stanley2}. There are also other possible extensions, see e.g. \cite{faridi2}.

In \cite{hh2}, Herzog and Hibi characterized all Cohen-Macaulay bipartite graphs $G$, showing that a bipartite graph is Cohen-Macaulay if and only if it can be associated to a poset $P$ under a suitable construction. In \cite{ehm}, Ene, Herzog and Mohammadi extended this construction introducing the clutter $\mathcal{C}_{P,d}$, for $d \geq 2$, and describing the maps of the resolution of the Alexander dual ideal $I(\mathcal{C}_{P,d})^*$ of $I(\mathcal{C}_{P,d})$, the so-called {\em cover ideal}. 


By using Betti splitting techniques introduced in \cite{ur}, in Theorem \ref{cmdip} we prove an explicit recursive formula for the graded Betti numbers of $I(\mathcal{C}_{P,d})^*$, extending a result by Francisco,  H\`a and Van Tuyl \cite[Theorem 3.8]{ur}. There are several preliminary results to the proof of the main theorem, but we think that Proposition \ref{splitinside} and Proposition \ref{final1} could be results of independent interest. In the case of a poset $P$ with a unique maximal element, the formula can be considerably simplified (see Corollary \ref{simpler}). In Corollary \ref{cmbip} we give another formulation of \cite[Theorem 3.8]{ur}, on Betti numbers of the cover ideal of a Cohen-Macaulay bipartite graph.

In Example \ref{dunce} we show that Betti splittings could {\em not exist} for the Alexander dual ideal $I_{\Delta}^*$ of a simplicial complex $\Delta$. Moreover we point out that there is no relation between the existence of a Betti splitting of the Stanley-Reisner ideal $I_{\Delta}$ and the existence of a Betti splitting of the Alexander dual ideal $I_{\Delta}^*$. Finally we give a sufficient condition that ensures the existence of an $x_i$-splitting for $I_{\Delta}^*$ (see Proposition \ref{weakVD}). \\


{\em Acknowledgements.} The last section of this paper was inspired by the summer school {\em Discrete Morse Theory and Commutative Algebra}, held at the Institut Mittag-Leffler in Stockholm in 2012. The author is grateful to the staff of the Institut, to the organizers Bruno Benedetti and Alexander Engstr\"{o}m, for their suggestions on Discrete Morse Theory and properties of simplicial complexes. Thanks to Paolo Sentinelli for helpful discussions. 

\section{Preliminaries}\label{pre}
Let $\mathbbm{k}$ be a field. Throughout this paper, $S$ denotes the polynomial ring $\mathbbm{k}[x_1,...,x_n]$. For a monomial ideal $I \subseteq S$, let $\beta_{i,j}(I)=\mathrm{dim}_{\mathbbm{k}}\mathrm{Tor}_i(I,\mathbbm{k})_j$ be the {\em graded Betti numbers of $I$} and $\beta_i(I)=\sum_{j \in \mathbb{N}}\beta_{i,j}(I)$ be the {\em i-th total Betti numbers of $I$}. If $I$ is generated in degree $d$, we say that $I$ has a {\em $d$-linear resolution} provided $\beta_{i,i+j}(I)=0$ for every $i \in \mathbb{N}$ and $j \neq d.$ When the context is clear, we simply say that $I$ has a linear resolution. 

Denote by $\mathrm{supp}(m)$ the set of variables dividing a monomial $m$ and by $G(I)$ the set of minimal monomial generators of $I.$ If $I$ is a square-free monomial ideal, let  $I^{*}$ be its Alexander dual ideal (see for instance \cite{hibi}).\\

The idea of Betti splitting has been introduced in \cite{eiker}, but in this paper we follow the more general setting of \cite{ur}.

\begin{definition}{\em (\cite[Definition 1.1]{ur})}
\em Let $I$, $J$ and $K$ be monomial ideals such that $I=J+K$ and $G(I)$ be the disjoint union of $G(J)$ and $G(K)$. Then $J+K$ is a {\em Betti splitting} of $I$ if $$\beta_{i,j}(I)=\beta_{i,j}(J)+\beta_{i,j}(K)+\beta_{i-1,j}(J \cap K), \text{ for every } i,j \in \mathbb{N}.$$ If this is the case $\beta_{i}(I)=\beta_{i}(J)+\beta_{i}(K)+\beta_{i-1}(J \cap K), \text{ for every } i \in \mathbb{N}.$
\end{definition} 

In some cases we focus on a special splitting of a monomial ideal $I$. Let $x_i$ a fixed variable of $S$. Define $M:=\{m \in G(I):x_i \text{   divides   }m\}$, $J:=(\frac{m}{x_i}:m \in M)$ and let $K$ be the ideal generated by the remaining monomials of $G(I)$. We call $I=x_iJ+K$ an {\em $x_i$-partition}. If this is a Betti splitting, we say that $I=x_iJ+K$ is an {\em $x_i$-splitting} of $I$.

Francisco, H\`a and Van Tuyl, in \cite{ur}, proved the following splitting results that we use extensively in this paper.

\begin{proposition}\label{urur}{\em (\cite[Corollary 2.4]{ur},\cite[Corollary 2.7]{ur})}
Let $I$, $J$ and $K$ be monomial ideals in $S$ such that $I=J+K$ and $G(I)$ be the disjoint union of $G(J)$ and $G(K)$. Let $x_i$ be a variable of $S$. Then
\begin{itemize}
\item[{\em (i)}] If $J$ and $K$ have a linear resolution, then $I=J+K$ is a Betti splitting of $I$. 
\item[{\em (ii)}] If $I=x_iJ+K$ is an $x_i$-partition and $J$ has a linear resolution, then $I=J+K$ is a Betti splitting of $I$.
\end{itemize}
\end{proposition}

Since we are using properties of ideals with linear quotients, we recall here some useful facts. Ideals with linear quotients have been introduced by Herzog and Takayama in \cite{hhh}. For further details see \cite{hibi} and \cite{jz}.

\begin{definition}
\em Let $I \subseteq S$ be a monomial ideal. We say that the ideal $I$ has {\em linear quotients} if there exists an order $u_1,...,u_m$ of the monomials of $G(I)$ such that, for every $2 \leq j \leq k$, the colon ideal $(u_1,...,u_{j-1}):u_j$ are generated by a subset of $\{x_1,...,x_n\}$.
\end{definition}

An ideal with linear quotients is componentwise linear (see \cite[Corollary 2.8]{jz}) and, if generated in a single degree, it has a linear resolution (see \cite{hibi2}).

The following characterization of square-free monomial ideals with linear quotients will be useful later (see for instance \cite[Corollary 8.2.4]{hibi}).

\begin{proposition}\label{linquot}
Let $I \subseteq S$ be a square-free monomial ideal. Then $I$ has linear quotients with respect to the order $u_1,u_2,...,u_m$ of the monomial generators if and only if for each $i$ and for all $j<i$ there exists a variable $x_h \in supp(u_j) \setminus supp(u_i)$ and an integer $k<i$ such that $supp(u_k) \setminus supp(u_i)=\{x_h\}.$
\end{proposition} 

We now recall some definitions about clutters, for further details see \cite{morey2}.

\begin{definition}
\em A {\em clutter} $\mathcal{C}=(V(\mathcal{C}),E(\mathcal{C}))$ is a pair where $V(\mathcal{C})$ is a finite set, called {\em vertex set}, and $E(\mathcal{C})$ is a collection of non-empty subsets of $V(\mathcal{C})$, called {\em edges}, such that if $e_1$ and $e_2$ are distint edges of $\mathcal{C}$, then $e_1 \not\subseteq e_2$. The clutter $\mathcal{C}$ is {\em $d$-uniform} if all edges have exactly $d$ vertices. A {\em vertex cover} of a clutter $\mathcal{C}$ is a subset $W \subseteq V(\mathcal{C})$ such that $e \cap W \neq \emptyset$, for every $e \in E(\mathcal{C})$. A vertex cover is {\em minimal} if none of its proper subsets is a vertex cover.
\end{definition}

For $d=2$ this is the definition of simple graph. In literature clutters are also known as {\em simple hypergraphs}.  

Let $\mathcal{C}$ be a clutter on the vertex set $V(\mathcal{C})=\{1,...,n\}$. The {\em edge ideal} of the $\mathcal{C}$  is defined by $$I(\mathcal{C}):=\left(\prod_{i \in e} x_i:e \in E(\mathcal{C})\right) \subseteq S.$$ The edges of a clutter can be seen as the facets of a simplicial complex. Then the edge ideal of a clutter coincides with the facet ideal introduced by Faridi in \cite{faridi}. 

The Alexander dual ideal $I(\mathcal{C})^*$ of $I(\mathcal{C})$ is called the {\em cover ideal} of $\mathcal{C}$. Its minimal monomial generators correspond to the vertex covers of $\mathcal{C}$. 

The clutter $\mathcal{C}$ is called {\em Cohen-Macaulay} if $S/I(\mathcal{C})$ is Cohen-Macaulay. 

\begin{ex}
\em Let $V=\{1,2,3,4\}$. The collection $\{\{1,2\},\{1,2,3\},\{3,4\}\}$ is {\em not} a clutter. The clutter $\mathcal{C}$ whose edge set is $E(\mathcal{C})=\{\{1,2,3\},\{3,4\}\}$ is not uniform. Its edge ideal in $\mathbbm{k}[x_1,...,x_4]$ is $I(\mathcal{C})=(x_1x_2x_3,x_3x_4).$ Its cover ideal is $I(\mathcal{C})^*=(x_3,x_1x_4,x_2x_4)$ and in fact the minimal vertex covers of $\mathcal{C}$ are $\{3\},\{1,4\},\{2,4\}$.
\end{ex}

\begin{definition}
\em Let $\mathcal{C}=(V(\mathcal{C}),E(\mathcal{C}))$ be a $d$-uniform clutter. We say that $\mathcal{C}$ is {\em $d$-partite} if there is a partition $V(\mathcal{C})=V_1 \mathcal{t} V_2 \mathcal{t} ... \mathcal{t} V_{d}$ such that all the edges of $\mathcal{C}$ have the form $\{j_1,j_2,...,j_d\}$, where $j_r \in V_r,$ for $1 \leq r \leq d$. 
\end{definition}

For $d=2$, this is the definition of bipartite graph.\\

Let $P=\{p_1,...,p_n\}$ be a finite {\em poset}, with partial order $\leq$. We label each element of $P$ in such a way that if $p_i<p_j$, then $i<j.$ A subset $C \subseteq P$ is called a {\em chain} if $C=p_{i_1}<p_{i_2}<\cdots<p_{i_d}$ is a totally ordered set with respect to the induced order. A {\em multichain} of length $d$ in $P$ is a chain $p_{i_1} \leq p_{i_2} \leq \cdots \leq p_{i_{d}}$, in which we allow repetitions. 

Let $$\mathrm{Max}(P):=\{p \in P:\text{ there is no } q \in P \text{ such that } p<q\}.$$ By our labeling convention, we have $p_n \in \mathrm{Max}(P).$

A {\em poset ideal} $\alpha$ in $P$ is a subset of $P$ with the following property: given $p \in \alpha$ and $q \in P$ such that $q \leq p$, then $q \in \alpha.$  We may consider each poset ideal as a subposet of $P$ with respect to the induced order. Notice that, given $p \in \mathrm{Max}(\alpha)$, then $\alpha \setminus \{p\}$ is again a poset ideal of $P$.

Given $p_1,...,p_k$ elements of $P$, denote by $\langle p_1,...,p_k \rangle$ the smallest poset ideal containing $p_1,...,p_k.$ Let $\mathbb{I}(P)$ be {\em the set of all poset ideals of $P$}. 

\begin{ex}\label{poset1}
\em Let $P=\{p_1,...,p_6\}$ be the poset with $6$ elements, whose Hasse diagram is given in Figure 1.

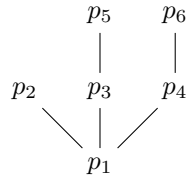
\begin{figure}[ht!]
\centering
\begin{tikzpicture}[scale=0.5]
  \node (a) at (0,0) {$p_1$};
  \node (b) at (-2,2) {$p_2$};
  \node (c) at (0,2) {$p_3$};
  \node (d) at (2,2) {$p_4$};
  \node (e) at (0,4) {$p_5$};
  \node (f) at (2,4) {$p_6$};
  \draw (f) -- (d) -- (a) -- (c) -- (e)
  (a) -- (b);
\end{tikzpicture}
\caption{A poset with 6 elements.} \label{exposet}
\end{figure} 

Examples of chains in $P$ are $p_1<p_5$ and $p_1<p_4<p_6$. Examples of multichains of length $4$ and $5$ are $p_2 \leq p_2 \leq p_2 \leq p_2$, $p_1 \leq p_1 \leq p_1 \leq p_6 \leq p_6$. Clearly $\mathrm{Max}(P)=\{p_2,p_5,p_6\}$. The subposet $\{p_1,p_4,p_6,p_2\}$ is a poset ideal of $P$ and it is equal to $\langle p_2,p_6 \rangle$. The subposet $\{p_2,p_3,p_4\}$ is {\em not} a poset ideal of $P$, because $p_1 \leq p_2$ but $p_1 \notin \{p_2,p_3,p_4\}$. 
\end{ex}

Let $P$ be a poset and $d \geq 2$ be an integer. We recall here the construction of the $d$-partite clutter $\mathcal{C}_{P,d}$.

Let $V=\{x_{ij}:1 \leq i \leq d,1 \leq j \leq n\}$ and $T=\mathbbm{k}[x_{ij}:1 \leq i \leq d,1 \leq j \leq n]$ a polynomial ring on $nd$ variables. The clutter $\mathcal{C}_{P,d}$ is the clutter with vertex set $V=V(\mathcal{C}_{P,d})$ whose edges are 
$$E(\mathcal{C}_{P,d}):=\{\{x_{1j_1},x_{2j_2},...,x_{dj_d}\}: p_{j_1} \leq p_{j_2} \leq \cdots \leq p_{j_d} \text{  is a multichain of length $d$ in $P$}\}.$$

Then the edge ideal of $\mathcal{C}_{P,d}$ is 

$$I(\mathcal{C}_{P,d})=\left(\prod_{r=1}^dx_{rj_r}: p_{j_1} \leq p_{j_2} \leq \cdots \leq p_{j_d} \text{  is a multichain of length $d$ in $P$}\right) \subseteq T.$$

We give a simple example of this construction.

\begin{ex}\label{poset2}
\em Let $P$ be the poset of Example \ref{poset1} and $d=3$. To avoid double indices, we set $x_j=x_{1j}$, $y_j=x_{2j}$ and $z_j=x_{3j}$. Let $T=\mathbbm{k}[x_1,x_2,x_3,y_1,y_2,y_3,z_1,z_2,z_3]$. To each multichain $p_i \leq p_j \leq p_k$ of length $3$ in $P$, we associate the monomial $x_iy_jz_k$. For instance, the multichain $p_1 \leq p_2 \leq p_2$ is associated with the monomial $x_1y_2z_2$. Then the edge ideal of $\mathcal{C}_{P,3}$ in $T$ is

\begin{center}
$I(\mathcal{C}_{P,3})=(x_1y_1z_1,x_2y_2z_2,x_3y_3z_3,x_4y_4z_4,x_5y_5z_5,x_6y_6z_6,x_1y_3z_5,x_1y_4z_6,x_1y_1z_2,x_1y_2z_2,$

$x_1y_1z_3,x_1y_3z_3,x_1y_1z_5,x_1y_5z_5,x_3y_3z_5,x_3y_5z_5,x_1y_1z_4,x_1y_4z_4,x_1y_1z_6,x_1y_6z_6,x_4y_4z_6,x_4y_6z_6).$
\end{center}
\end{ex}

Let $V_i:=\{x_{ij}:1 \leq j \leq n\},$ for $1 \leq i \leq d.$ Then $V(\mathcal{C}_{P,d})=V_1 \mathcal{t} V_2 \mathcal{t} ... \mathcal{t} V_{d}$. By construction, $\mathcal{C}_{P,d}$ is $d$-partite. Note that for $d=2$ this is the construction of a bipartite graph starting from a poset $P$, given by Herzog and Hibi (see \cite{hibi} and \cite{hh2}). For $d \geq 3$, {\em not} all Cohen-Macaulay $d$-uniform $d$-partite clutters arise from a poset (see e.g. \cite[Example 3.4]{morey}).\\

Following \cite[Theorem 1.1]{ehm}, we describe explicitly the generators of the ideal $I(\mathcal{C}_{P,d})^*$. 

A {\em poset multideal} of degree $d$ in $P$ is a $d$-tuple $\underline{\bm{\alpha}}=(\alpha_i)_{1 \leq i \leq d}$, such that 

$$\alpha_i \in \mathbb{I} \! \left(P \setminus \bigcup_{j=1}^{i-1}\alpha_j\right) \text{\quad and \quad} \alpha_d=P \setminus \left(\bigcup_{j=1}^{d-1}\alpha_j\right).$$

Notice that $\bigcup_{i=1}^d \alpha_i=P$.


Let $\underline{\bm{\alpha}}$ be a poset multideal of degree $d$ in $P.$ We define the monomial $$u_{\underline{\bm{\alpha}}}:=\prod_{i=1}^d\prod_{p_j \in \alpha_i}x_{ij}.$$ Each monomial $u_{\underline{\bm{\alpha}}}$ has degree $n=|P|.$ Here and in what follows, $|\cdot|$ denotes the cardinality of a finite set. By \cite[Theorem 1.1]{ehm}, we have $$I(\mathcal{C}_{P,d})^*=(u_{\underline{\bm \alpha}})_{\underline{\bm \alpha} \in \mathbb{I}_d(P)}.$$ To simplify the notation, from now on we denote by $H_{P,d}$ the ideal $I(\mathcal{C}_{P,d})^*$. 


\begin{ex}
\em Let $P$ be the poset of Example \ref{poset1}. Examples of poset multideal of degree $3$ and $4$ in $P$, respectively, are 
$$\underline{\bm{\alpha}}=(\{p_1,p_3\},\{p_2,p_5\},\{p_4,p_6\}) \text{   and   } \underline{\bm{\gamma}}=(\{p_1,p_2,p_4\},\emptyset,\{p_3,p_6\},\{p_5\}).$$ Consider the notation of Example \ref{poset2} to avoid double indices, setting $t_j=x_4j$. We have the monomials $u_{\underline{\bm{\alpha}}}=x_1x_3y_2y_5z_4z_6$ and $u_{\underline{\bm{\gamma}}}=x_1x_2x_4z_3z_6t_5$.
\end{ex}

In \cite[Theorem 2.4]{ehm}, the authors proved that $H_{P,d}$ is a weakly polymatroidal ideal. Weakly polymatroidal ideals have been introduced by Hibi and Kokubo in \cite{kokubo}. From \cite[Theorem 1.3]{moha}, it follows that a weakly polymatroidal ideal has linear quotients. Then the ideal $H_{P,d}$ has linear quotients. In particular it has a linear resolution. Hence $S/I(\mathcal{C}_{P,d})$ is Cohen-Macaulay, by \cite[Theorem 3]{eagrein}.



\section{The recursive formula}



Our main result is a recursive formula for the Betti numbers of $H_{P,d}$, extending a recent result due to Francisco, H\`a and Van Tuyl, \cite[Theorem 3.8]{ur}. We recall that the ideal $H_{P,d}$ is generated in a single degree $n=|P|$ and that it has a linear resolution.

\begin{theorem}\label{cmdip}
Let $P \neq \emptyset$ be a finite poset, $d \geq 2$ be an integer. Then, for $i \geq 0$, $$\beta_i(H_{P,d})=\sum_{\alpha \in \mathbb{I}(P)} \left[\sum_{h=0}^{|\mathrm{Max}(\alpha)|}\binom{|\mathrm{Max}(\alpha)|}{h}\beta_{i-h}(H_{\alpha,d-1}) \right].$$
\end{theorem}

Before proving the theorem, we explain the situation in some degenerate cases. 

\begin{remark}\label{degenerate} 
\begin{itemize}
\item[{\em (i)}] If $P=\emptyset$, then $H_{\emptyset,d}=0$.
\item[{\em (ii)}] If $P=\{p_1\}$, then $H_{P,d}=(x_{11},x_{21},\dots,x_{d1})$. It is well-known that this ideal has a linear resolution and its Betti numbers are given by $$\beta_i(H_{\{p_1\},d})=\binom{d}{i+1}, \text{   for   } i \geq 0.$$ 
\item[{\em (iii)}] If $P \neq \emptyset$ is a finite poset and $d=1$, then $H_{P,1}$ is generated by a single monomial $m=\prod_{j=1}^nx_{1j}$. Then $\beta_{0}(H_{P,1})=1$ and $\beta_{i}(H_{P,1})=0$ for each $i \geq 1$.
\end{itemize}
\end{remark}

We now prove some technical results, which will be useful later. 

\begin{proposition}\label{splitinside}
Let $I \subseteq S=\mathbbm{k}[x_1,...,x_n]$ be a square-free monomial ideal with a $d$-linear resolution and $\M=(x_1,...,x_n)$. Let $k$ be a positive integer and $\N=(y_1,...,y_k) \subseteq \mathbbm{k}[y_1,...,y_k]$. Then the ideal $\N I \subseteq R=\mathbbm{k}[x_1,...,x_n,y_1,...,y_k]$ has a $(d+1)$-linear resolution and $$\beta_i(\N I)=\sum_{s=0}^{k-1}\binom{k}{s+1}\beta_{i-s}(I), \text{  for every  } i \geq 0.$$
\end{proposition}

{\em Proof.} We first prove that the ideal $\N I$ has a $(d+1)$-linear resolution. Notice that $y_1,...,y_k$ is an $IR$-regular sequence. By \cite[Theorem 2.2]{aldo}, $\N I$ has a $(d+1)$-linear resolution, since $I$ has a $d$-linear resolution.

To prove the formula we proceed by induction on $k.$ If $k=1$, we have nothing to show, since $\beta_i(\N I)=\beta_i(y_1I)=\beta_i(I)$, for every $i \geq 0.$ Let $k>1.$ One has $\N I=(y_1,...,y_{k-1})I+y_kI$. Notice that $G(\N I)$ is the disjoint union of $G((y_1,...,y_{k-1})I)$ and $G(y_kI)$. By Proposition \ref{urur} (ii), this decomposition is a Betti splitting of $\N I$, because $y_kI$ has a $(d+1)$-linear resolution. Then $$\beta_{i}(\N I)=\beta_{i}((y_1,...,y_{k-1})I)+\beta_{i}(y_kI)+\beta_{i-1}(y_k(y_1,...,y_{k-1})I), \text{  for every  } i \geq 0.$$ 

Since $\beta_{i}(y_kI)=\beta_i(I)$ and $\beta_{i-1}(y_k(y_1,...,y_{k-1})I)=\beta_{i-1}((y_1,...,y_{k-1})I)$, for every $i \geq 0$, by induction $$\beta_i(\N I)=\sum_{s=0}^{k-2}\binom{k-1}{s+1}\beta_{i-s}(I)+\beta_{i}(I)+\sum_{s=0}^{k-2}\binom{k-1}{s+1}\beta_{i-1-s}(I)=\sum_{s=0}^{k-1}\binom{k}{s+1}\beta_{i-s}(I). \qed$$

\begin{proposition}\label{final1}
Let $d,t \geq 1$ be integers. Let $\{I_k\}_{1 \leq k \leq t}$ be an ordered collection of monomial ideals with a $d$-linear resolution in $S=\mathbbm{k}[x_1,...,x_n]$. Set $I:=\sum_{k=1}^tI_k$. Assume $G(I)$ is the disjoint union of $\{G(I_k)\}_{1 \leq k \leq t}$ and that:
\begin{itemize}
\item[{\em (i)}] for every $2 \leq j \leq t$, the ideal $\sum_{k=1}^j I_k$ have a $d$-linear resolution;
\item[{\em (ii)}] for every $2 \leq j \leq t$, there exists an ideal $\N_j$, generated by a subset of the variables $x_1,...,x_n$, such that $I_j \cap \left( \sum_{k=1}^{j-1}I_k \right)=\N_j I_j$.
\end{itemize}
Then $$\beta_i(I)=\beta_i(I_1)+\sum_{k=2}^{t} \left[\sum_{s=0}^{|G(\N_k)|}\binom{|G(\N_k)|}{s}\beta_{i-s}(I_k) \right], \text{ for  } i \geq 0.$$
\end{proposition}

\begin{proof}
We may assume $t>1$, since for $t=1$ the formula is clear. Consider the following splitting of $I$: $$I=\sum_{k=1}^{t-1} I_k+I_t.$$ By our assumption, $G(I)$ is the disjoint union of $G(I_t)$ and $G(\sum_{k=1}^{t-1} I_k)$. Moreover the ideals $I_t$ and $\sum_{k=1}^{t-1} I_k$ have a $d$-linear resolution by (i), with $j=t-1$. By Proposition \ref{urur} (i), the splitting above is a Betti splitting of $I$. By assumption (ii), we have $I_t \cap \sum_{k=1}^{t-1} I_k=\N_t I_t$. Then $$\beta_i(I)=\beta_i(I_t)+\beta_i \left(\sum_{k=1}^{t-1} I_k \right)+\beta_{i-1}(\N_t I_t), \text{  for every  } i \geq 0.$$ By Proposition \ref{splitinside}, we get $$\beta_i(I)=\beta_i(I_t)+\beta_i \left(\sum_{k=1}^{t-1} I_k \right)+\sum_{h=0}^{|G(\N_t)|-1}\binom{|G(\N_t)|}{h+1}\beta_{i-1-h}(I_t), \text{  for every  } i \geq 0.$$  We conclude by setting $s=h+1$ and by induction on $t$.
\end{proof}

Let $P=\{p_1,...,p_n\}$ be a poset and recall that, for our labeling convention, $p_n \in \mathrm{Max}(P)$. Let $\psi=\langle p_n \rangle$ the smallest poset ideal of $P$ containing $p_n$. Define $$\mathbb{I}(P,p_n):=\{\alpha \in \mathbb{I}(P):p_n \in \alpha\}.$$ 

Let $d \geq 2$ be an integer. For every $\alpha \in \mathbb{I}(P,p_n)$, we define the ideal $$J_{\alpha}:=\left( \prod_{p_i \notin \alpha}x_{di} \right)H_{\alpha,d-1}.$$ In order to define the ideal $J_{\alpha}$, we consider $\alpha$ as a subposet of $P$ with respect to the induced order. 

Notice that $J_{\alpha}$ is generated in a single degree $n=|P|$. Since the ideal $H_{\alpha,d-1}$ has a linear resolution, it follows that $J_{\alpha}$ has a linear resolution. 
\vspace{2mm}
We now give an order of the ideals $\{J_{\alpha}\}_{\alpha \in \mathbb{I}(P,p_n)}$. Let $\alpha_1,\alpha_2 \in \mathbb{I}(P,p_n)$, with $\alpha_1 \neq \alpha_2$. We set $\alpha_1 \triangleleft \alpha_2$ if one of the following conditions is satisfied:

\begin{itemize}
\item[(i)] $|\alpha_1|<|\alpha_2|$;
\item[(ii)] $|\alpha_1|=|\alpha_2|$ and $\mathrm{max}\{j:p_j \in \alpha_1 \setminus \alpha_2\}<\mathrm{max}\{j:p_j \in \alpha_2 \setminus \alpha_1\}$. 
\end{itemize}  

Note that, with respect to the above order, $\psi \ \underline{\triangleleft} \ \alpha$, for every $\alpha \in \mathbb{I}(P,p_n)$.

\begin{ex}\label{poset3}
\em Let $P$ be the poset of Example \ref{poset1}. Clearly $p_6 \in \mathrm{Max}(P)$ and $\psi=\langle p_6 \rangle=\{p_1,p_4,p_6\}$. The poset ideals containing $p_6$ are $$\mathbb{I}(P,p_n):=\{\{p_1,p_4,p_6\},\{p_1,p_4,p_6,p_2\},\{p_1,p_4,p_6,p_3\},\{p_1,p_4,p_6,p_2,p_3\},\{p_1,p_4,p_6,p_3,p_5\},P\}.$$ Consider $\alpha=\{p_1,p_4,p_6,p_2\}$, $d=3$ and 
set of $P$ with respect of the induced order:
\begin{center} 
$(\emptyset,P),(\{p_1\},\{p_2,p_4,p_6\}),(\{p_1,p_2\},\{p_4,p_6\}),(\{p_1,p_4\},\{p_2,p_6\}),$
$(\{p_1,p_2,p_4\},\{p_6\}),(\{p_1,p_4,p_6\},\{p_2\}),(P,\emptyset).$ 
\end{center}
Then $J_{\alpha}=z_3z_5(y_1y_2y_4y_6,x_1y_2y_4y_6,x_1x_2y_4y_6,x_1x_4y_2y_6,x_1x_2x_4y_6,x_1x_4x_6y_2,x_1x_2x_4x_6).$

Notice that the given order of the elements of $\mathbb{I}(P,p_n)$ is the order $\triangleleft$ defined above.
\end{ex}

\begin{proposition}\label{partialsumlinear}
Let $P$ be a poset, $p_n \in \mathrm{Max}(P)$ and fix $\alpha \in \mathbb{I}(P,p_n)$. Then the ideal $\displaystyle{J:=\sum_{\phi \in \mathbb{I}(P,p_n),\phi \underline{\triangleleft} \alpha}J_{\phi}}$ has linear quotients. In particular $J$ has a linear resolution.
\end{proposition}

\begin{proof}
Assume $\alpha=\psi$. Then $J=J_{\psi}$. This ideal has linear quotients, since $H_{\psi,d-1}$ has linear quotients. 

We may assume $\alpha \neq \psi$. To prove that $J$ has linear quotients, we use Proposition \ref{linquot}. We want to define an ordering $<_J$ of the monomials in $G(J)$. For each $\phi \in \mathbb{I}(P,p_n)$, $J_{\phi}$ has linear quotients with respect to an ordering $<_{\phi}$ of the monomials in $G(J_{\phi})$. 

Let $u,v \in G(J)$. Then $u \in G(J_{\phi_1})$ and $v \in G(J_{\phi_2})$, for some $\phi_1,\phi_2 \in \mathbb{I}(P,p_n)$. We set $u<_J v$ if one of the following conditions is satisfied:

\begin{itemize}
\item[(i)] $\phi_1 \neq \phi_2$ and $\phi_1 \triangleleft \phi_2$;
\item[(ii)] $\phi_1=\phi_2=\phi$ and $u <_{\phi} v$. 
\end{itemize}  

With respect to the above notation, let $u,v \in G(J)$ such that $u <_J v$. 

Assume first $\phi_1=\phi_2=\phi$. Since $J_{\phi}$ has linear quotients, by Proposition \ref{linquot}, there exists a variable $x \in \mathrm{supp}(u) \setminus \mathrm{supp}(v)$ and a monomial $w \in G(J_{\phi})$, with $w <_{\phi} v$ such that $\mathrm{supp}(w) \setminus \mathrm{supp}(v)=\{x\}$. By definition of $<_J$, we have $w <_J v$ and we are done.

Assume now $\phi_1 \neq \phi_2$. We have $\phi_1 \triangleleft \phi_2$. Then there exist $p_j \in \mathrm{Max}(\phi_2) \setminus \phi_1$. In fact assume, by contradiction, that $\mathrm{Max}(\phi_2) \subseteq \phi_1$. It follows that $\phi_2 \subseteq \phi_1$, since $\phi_1$ is a poset ideal of $P$. This is a contradiction, since $|\phi_1| \leq |\phi_2|$ and $\phi_1 \neq \phi_2$. Then $p_j \neq p_n$, because $p_n \in \phi_1$. 

Since $p_j \in \phi_2 \setminus \phi_1$, then $x_{dj} \in \mathrm{supp}(u) \setminus \mathrm{supp}(v)$. By definition of $J_{\phi_2}$, there exists a poset multideal $\underline{\bm{\gamma}}=(\gamma_1,...,\gamma_{d-1})$ of degree $d-1$ in $\phi_2$ such that  $v=\left( \prod_{p_i \notin \phi_2}x_{di} \right)u_{\underline{\bm{\gamma}}}.$ Since $p_j \in \phi_2$, there exists $1 \leq k \leq d-1$ such that $p_j \in \gamma_k$. Consider the monomial $$w=\left( \prod_{p_i \in (P \setminus \phi_2) \cup \{p_j\}} x_{di} \right)u_{\underline{\bm{\gamma}}\setminus \{p_j\}},$$ where $\underline{\bm{\gamma}}\setminus \{p_j\}:=(\gamma_1,...,\gamma_k \setminus \{p_j\},...,\gamma_{d-1})$. Clearly $\phi_2 \setminus \{p_j\}$ is a poset ideal of $P$, because $p_j \in \mathrm{Max}(\phi_2)$ and we have $\phi_2 \setminus \{p_j\} \in \mathbb{I}(P,p_n)$, since $\phi_2 \in \mathbb{I}(P,p_n)$. Moreover $\underline{\bm{\gamma}}\setminus \{p_j\}$ is a poset multideal of degree $d-1$ in $\phi_2 \setminus \{p_j\}$. Then $w \in G(J_{\phi_2 \setminus \{p_j\}})$. Since $|\phi_2 \setminus \{p_j\}|<|\phi_2|$, one has $\phi_2 \setminus \{p_j\} \triangleleft \phi_2$ and $w <_J v$. From the fact that $\mathrm{supp}(w) \setminus \mathrm{supp}(v)=\{x_{dj}\}$, we conclude by Proposition \ref{linquot}.
\end{proof}

In the next lemma, we summarize some important properties of the ideals $J_{\alpha}$ and $H_{\alpha,d}$, for $\alpha \in \mathbb{I}(P,p_n)$. We recall that the intersection $J \cap K$ of two monomial ideals is generated by $\{\mathrm{lcm}(u,v):u \in G(J),v \in G(K)\}$, where $\mathrm{lcm}(u,v)$ denotes the least common multiple of the monomials $u$ and $v$. Clearly $\mathrm{supp}(\mathrm{lcm}(u,v))=\mathrm{supp}(u) \cup \mathrm{supp}(v)$, see \cite{hibi}.

\begin{lemma}\label{inclusion}
Let $P$ be a poset, $p_n \in \mathrm{Max}(P)$, $\alpha, \gamma \in \mathbb{I}(P,p_n)$, $d \geq 2$ an integer. Then 
\begin{enumerate}
\item[{\em (i)}] $J_{\alpha} \subseteq H_{P\setminus \{p_n\},d}$;
\item[{\em (ii)}] $H_{\alpha,d} \cap H_{\gamma,d}=H_{\alpha \cup \gamma,d}$ and, in particular, $H_{\alpha \cup \gamma,d} \subseteq H_{\alpha,d}$; 
\item[{\em (iii)}] $J_{\alpha} \cap J_{\gamma}=\left(\prod_{p_i \in P \setminus (\alpha \cap \gamma)}x_{di}\right)H_{\alpha \cup \gamma,d-1}$.
\end{enumerate} 
\end{lemma}

\begin{proof} First we prove (i). Let $u \in G(J_{\alpha}).$ By definition there exists a poset multideal $(\gamma_1,...,\gamma_{d-1})$ of degree $d-1$ in $\alpha$, such that $u=u_{\underline{\bm{\gamma}}}$, where $\underline{\bm{\gamma}}=(\gamma_1,...,\gamma_{d-1},P \setminus \alpha)$. Since $p_n \in \alpha$, then there exists $1 \leq i \leq d-1$ such that $p_n \in \gamma_i.$ Note that $\underline{\bm{\gamma}} \setminus \{p_n\}:=(\gamma_1,...,\gamma_i\setminus \{p_n\},...,\gamma_{d-1},P \setminus {\alpha})$ is a poset multideal of degree $d$ in $P \setminus \{p_n\}$, because $p_n \in \mathrm{Max}(P)$. Hence $u_{\underline{\bm{\gamma}}}=x_{di}u_{\underline{\bm{\gamma}} \setminus \{p_n\}} \in H_{P\setminus \{p_n\},d}.$

Now we prove (ii). We first show the inclusion from left to right. Let $m \in G(H_{\alpha,d} \cap H_{\gamma,d}).$ There exist $\underline{\bm{\alpha}}=(\alpha_1,...,\alpha_d)$ and $\underline{\bm{\gamma}}=(\gamma_1,...,\gamma_d)$  poset multideals of degree $d$ in $\alpha$ and $\gamma$, respectively, such that $m=\mathrm{lcm}(u_{\underline{\bm{\alpha}}},u_{\underline{\bm{\gamma}}})$. Define 

$$\phi_i:=(\alpha_i \cup \gamma_i) \setminus \bigcup_{j=1}^{i-1}(\alpha_j \cup \beta_j), \text{   for $1 \leq i \leq d-1$;       } \text{\quad and \quad} \phi_d:=(\alpha \cup \gamma) \setminus \bigcup_{j=1}^{d-1}\phi_j.$$ 

We claim that $\underline{\bm{\phi}}=(\phi_1,...,\phi_d)$ is a poset multideal of $\alpha \cup \gamma$. Fix $1 \leq i \leq d-1$. Notice that $\phi_i \subseteq (\alpha \cup \gamma) \setminus \bigcup_{j=1}^{i-1}\phi_j$, because $\bigcup_{j=1}^{i-1}\phi_j=\bigcup_{j=1}^{i-1}(\alpha_j \cup \beta_j)$. If $\phi_i=\emptyset$, we have nothing to prove, then we may assume $\phi_i \neq \emptyset$. Let $p \in \phi_i$ and $q \in (\alpha \cup \gamma) \setminus \bigcup_{j=1}^{i-1}\phi_j$, with $q \leq p$. We prove that $q \in \phi_i$. Assume first $p \in \alpha_i \subseteq \alpha \setminus \bigcup_{j=1}^{i-1}\alpha_j$. Since $\alpha$ is a poset ideal of $P$ and $p \in \alpha$, we have $q \in \alpha$. On the other hand, $q \in \alpha \setminus \bigcup_{j=1}^{i-1}\phi_j \subseteq  \alpha \setminus \bigcup_{j=1}^{i-1}\alpha_j$ and $\alpha_i$ is a poset ideal of $ \alpha \setminus \bigcup_{j=1}^{i-1}\alpha_j$, then $q \in \alpha_i$. Hence $q \in \phi_i$. The proof is the same if $p \in \gamma_i$. 

From the fact that $\phi_i \subseteq \alpha_i \cup \gamma_i$, for every $1 \leq i \leq d$, we have that $\mathrm{supp}(u_{\underline{\bm{\phi}}}) \subseteq \mathrm{supp}(m)$. Then $m \in H_{\alpha \cup \gamma,d}$, since $u_{\underline{\bm{\phi}}}$ and $m$ are square-free monomials.

We now prove the reverse inclusion. Let $u_{\underline{\bm{\phi}}} \in G(H_{\alpha \cup \gamma,d}),$ where $\underline{\bm{\phi}}=(\phi_1,...,\phi_d)$ is a poset multideal of degree $d$ in $\alpha \cup \gamma.$ Let $\phi_i^{\alpha}:=\phi_i \cap \alpha$ and $\phi_i^{\gamma}:=\phi_i \cap \gamma$, for $1 \leq i \leq d.$ Clearly $\phi_i=\phi_i^{\alpha} \cup \phi_i^{\gamma}.$ We prove that $(\phi_1^{\alpha},...,\phi_d^{\alpha})$ is a poset multideal of degree $d$ in $\alpha$, proving that $\phi_i^{\alpha}$ is a poset ideal of $\alpha \setminus \cup_{j=1}^{i-1}(\phi_j^{\alpha})$ for every integer $1 \leq i \leq d-1$. Let $p \in \phi_i^{\alpha}$ and $q \in \alpha \setminus \cup_{j=1}^{i-1}(\phi_j^{\alpha})$ such that $q \leq p.$ Since in particular $q \in \alpha \setminus \cup_{j=1}^{i-1}(\phi_j),$ then $q \in \phi_i.$ Since $q \in \alpha$, then $q \in \phi_i^{\alpha}$. The proof for $\gamma$ is the same. Then $u_{\underline{\bm{\phi}}}=\mathrm{lcm}(u_{\phi_i^{\alpha}},u_{\phi_i^{\gamma}}) \in H_{\alpha,d} \cap H_{\gamma,d}.$

Eventually we prove (iii). We have $$J_{\alpha} \cap J_{\gamma}=\mathrm{lcm} \left(\left(\prod_{p_i \in P \setminus \alpha}x_{di}\right),\left(\prod_{p_i \in P \setminus \gamma}x_{di}\right)\right)H_{\alpha,d-1} \cap H_{ \gamma,d-1}.$$ By part (ii) and since $(P \setminus \alpha) \cup (P \setminus \beta)=P \setminus (\alpha \cap \beta)$, the result follows.\end{proof}



\begin{proposition}\label{final}
Let $P=\{p_1,...,p_n\}$ be a poset. Then $$\beta_i \left(\sum_{\alpha \in \mathbb{I}(P,p_n)}J_{\alpha} \right)=\sum_{\alpha \in \mathbb{I}(P,p_n)} \left[\sum_{s=0}^{|\mathrm{Max}(\alpha)|-1}\binom{|\mathrm{Max}(\alpha)|-1}{s}\beta_{i-s}(H_{\alpha,d-1}) \right], \text{ for every } i \geq 0.$$ 
\end{proposition}

\begin{proof} To prove the formula, we use Proposition \ref{final1}. We consider the order $\triangleleft$ defined above on the ideals $\{J_{\alpha}\}_{\alpha \in \mathbb{I}(P,p_n)}$.

We recall that the ideals $\{J_{\alpha}\}_{\alpha \in \mathbb{I}(P,p_n)}$ have a linear resolution and that $G(\sum_{\alpha \in \mathbb{I}(P,p_n)}J_{\alpha})$ is the disjoint union of $\{G(J_{\alpha})\}_{\alpha \in \mathbb{I}(P,p_n)}$. By Proposition \ref{partialsumlinear}, the ideals $\sum_{\phi \underline{\triangleleft} \alpha}J_{\phi}$ have a linear resolution, for every $\alpha \in \mathbb{I}(P,p_n)$. Then condition (i) of Proposition \ref{final1} is fulfilled. 

We verify now condition (ii) of Proposition \ref{final1}. Fix $\alpha \in \mathbb{I}(P,p_n)$, with $\alpha \neq \psi$. We define the following set: $$R_{\alpha}:=\{\gamma \in \mathbb{I}(P,p_n):\gamma \subseteq \alpha \text{  and   } |\gamma|=|\alpha|-1\}.$$ We claim that 
\begin{equation}\label{jj}
J_{\alpha} \cap \left(\sum_{\phi \in \mathbb{I}(P,p_n),\phi \triangleleft \alpha}J_{\phi}\right)=J_{\alpha} \cap \left(\sum_{\gamma \in R_{\alpha}}J_{\gamma}\right).
\end{equation}

The inclusion from right to left is clear, since for every $\gamma \in R_{\alpha}$, we have $\gamma \in \mathbb{I}(P,p_n)$, $|\gamma|<|\alpha|$ and then $\gamma \triangleleft \alpha$. 

To prove the other inclusion, let $\phi \in \mathbb{I}(P,p_n)$, such that $\phi \triangleleft \alpha$. Since $|\phi| \leq |\alpha|$ and $\phi \neq \alpha$, then there exists $p_j \in \mathrm{Max}(\alpha) \setminus \phi$ (by the same argument in the proof of Proposition \ref{partialsumlinear}). 

Since $p_n \in \phi$, then $p_{j} \neq p_n.$ Then the poset ideal $\gamma:=\alpha \setminus \{p_{j}\}$ is an element of $\mathbb{I}(P,p_n)$, such that $\gamma \subseteq \alpha$ and $|\gamma|=|\alpha|-1$, i.e. $\gamma \in R_{\alpha}$. Notice that $\alpha \cap \phi \subseteq \gamma$, then $P \setminus (\alpha \cap \phi) \supseteq P \setminus \gamma.$ From this and from Lemma \ref{inclusion} (ii)-(iii), it follows that $$J_{\alpha} \cap J_{\phi}=\left(\prod_{p_i \in P \setminus (\alpha \cap \phi)} x_{di}\right) H_{\alpha \cup \phi,d-1} \subseteq \left(\prod_{p_i \in P \setminus \gamma}x_{di}\right) H_{\alpha,d-1}=J_{\alpha} \cap J_{\gamma}.$$ 



 
So we proved (\ref{jj}). 

For $\alpha \in \mathbb{I}(P,p_n)$ and $\gamma \in R_{\alpha}$, let $p_{j_{\gamma}}$ such that $\gamma=\alpha \setminus \{p_{j_{\gamma}}\}.$ By (\ref{jj}) and Lemma \ref{inclusion} (iii) it follows $$J_{\alpha} \cap \left(\sum_{\phi \in \mathbb{I}(P,p_n),\phi \triangleleft \alpha}J_{\phi}\right)=\sum_{\gamma \in R_{\alpha}}(J_{\alpha} \cap J_{\gamma})=\sum_{\gamma \in R_{\alpha}}\left(\prod_{p_i \in (P \setminus \alpha) \cup \{p_{j_{\gamma}}\}}x_{di}\right)H_{\alpha,d-1}=\left(\sum_{\gamma \in R_{\alpha}}x_{dj_{\gamma}}\right)J_{\alpha}.$$ 

So we proved that condition (ii) in Proposition \ref{final1} is fulfilled. One has $|R_{\alpha}|=|\mathrm{Max}(\alpha)|-1$ by construction. Recall that $\beta_{i}(J_{\alpha})=\beta_i(H_{\alpha,d-1}),$ for $i \geq 0$. Then by Proposition \ref{final1}, for every $i \geq 0$ $$\beta_i \left(\sum_{\alpha \in \mathbb{I}(P,p_n)}J_{\alpha} \right)=\beta_i(H_{\psi,d-1})+\sum_{\alpha \in \mathbb{I}(P,p_n), \alpha \neq \psi} \left[\sum_{s=0}^{|\mathrm{Max}(\alpha)|-1}\binom{|\mathrm{Max}(\alpha)|-1}{s}\beta_{i-s}(H_{\alpha,d-1}) \right].$$ Since the poset ideal $\psi$ is such that $\mathrm{Max}(\psi)=\{p_n\}$, the result follows. \end{proof}

Now we are able to prove Theorem \ref{cmdip}.
\vskip2mm
{\em Proof of Theorem \ref{cmdip}}. We proceed by induction on $n=|P| \geq 1$. For $n=1$, the formula holds by Remark \ref{degenerate}. 

We may assume $n \geq 2$. We prove that $H_{P,d}=x_{dn}H_{P \setminus \{p_n\},d}+\sum_{\alpha \in \mathbb{I}(P,p_n)}J_{\alpha}.$ 

Let $u \in G(H_{P,d})$ and $\underline{\bm{\phi}}=(\phi_1,...,\phi_d)$ be a poset multideal of degree $d$ in $P$ such that $u=u_{\underline{\bm{\phi}}}.$ Assume first $p_n \in \phi_d$, then $x_{dn}$ divides $u_{\underline{\bm{\phi}}}.$ Notice that $\underline{\bm{\phi}} \setminus \{p_n\}=(\phi_1,...,\phi_{d-1},\phi_d \setminus \{p_n\})$ is a poset multideal of degree $d$ in $P \setminus \{p_n\}$ and $u=x_{dn}u_{\underline{\bm{\phi}} \setminus \{p_n\}} \in x_{dn}H_{P \setminus \{p_n\},d}$. 

Assume $p_n \in \phi_i$, for $1 \leq i<d.$ Let $\alpha:=\bigcup_{j=1}^{d-1}\phi_j,$ then $\alpha \in \mathbb{I}(P,p_n).$ Hence $(\phi_1,...,\phi_{d-1})$ is a poset multideal of degree $d-1$ in $\alpha$ and $u_{\underline{\bm{\phi}}} \in J_{\alpha}.$ 

Conversely, let $u \in G(x_{dn}H_{P \setminus \{p_n\},d})$. Let $\underline{\bm{\delta}}=(\delta_1,...,\delta_d)$ be a poset multideal of degree $d$ in $P \setminus \{p_n\}$ such that $u=x_{dn}u_{\underline{\bm{\delta}}}.$ Since $(\delta_1,...,\delta_d \cup \{p_n\})$ is a poset multideal of degree $d$ in $P,$ one has $u \in H_{P,d}.$ Let $u \in G(J_{\alpha}),$ such that $\alpha \in \mathbb{I}(P,p_n).$ Let $\underline{\bm{\gamma}}=(\gamma_1,...,\gamma_{d-1})$ a poset multideal of degree $d-1$ in $\alpha$ such that $u=\left( \prod_{p_i \notin \alpha}x_{di} \right)u_{\underline{\bm{\gamma}}}$. One has $u \in G(H_{P,d})$ since $(\gamma_1,...,\gamma_{d-1},P \setminus \alpha)$ is a poset multideal of degree $d$ in $P.$ 

Recall that the ideal $H_{P \setminus \{p_n\},d}$ has a linear resolution. Clearly $G(H_{P,d})$ is the disjoint union of $G(x_{dn}H_{P \setminus \{p_n\},d})$ and $G(\sum_{\alpha \in \mathbb{I}(P,p_n)}J_{\alpha})$. By Proposition \ref{urur} (ii), the splitting above is an $x_{dn}$-splitting. Notice that $x_{dn}H_{P \setminus \{p_n\},d}$ has a linear resolution and $\beta_{i}(x_{dn}H_{P \setminus \{p_n\},d})=\beta_{i}(H_{P \setminus \{p_n\},d}),$ for every $i \geq 0$. By Lemma \ref{inclusion} (i) $$x_{dn}H_{P \setminus \{p_n\},d} \cap \left(\sum_{\alpha \in \mathbb{I}(P,p_n)}J_{\alpha}\right)\!=\!\!\sum_{\alpha \in \mathbb{I}(P,p_n)}x_{dn}(H_{P \setminus \{p_n\},d} \cap J_{\alpha})=\!\!\sum_{\alpha \in \mathbb{I}(P,p_n)}x_{dn}J_{\alpha}=x_{dn} \left(\sum_{\alpha \in \mathbb{I}(P,p_n)}J_{\alpha} \right).$$ Then  
\begin{equation}\label{ok}
\beta_i(H_{P,d})=\beta_i(H_{P \setminus \{p_n\},d})+\beta_i \left( \sum_{\alpha \in \mathbb{I}(P,p_n)}J_{\alpha} \right)+\beta_{i-1}\left( \sum_{\alpha \in \mathbb{I}(P,p_n)}J_{\alpha} \right).
\end{equation} 

By Proposition \ref{final} one has, for every $i \geq 0$, $$\beta_i \left( \sum_{\alpha \in \mathbb{I}(P,p_n)}J_{\alpha} \right)+\beta_{i-1}\left( \sum_{\alpha \in \mathbb{I}(P,p_n)}J_{\alpha} \right)=\sum_{\alpha \in \mathbb{I}(P,p_n)} \left[\sum_{s=0}^{|\mathrm{Max}(\alpha)|}\binom{|\mathrm{Max}(\alpha)|}{s}\beta_{i-s}(H_{\alpha,d-1}) \right].$$ Notice that $\mathbb{I}(P)=\mathbb{I}(P \setminus \{p_n\}) \cup \mathbb{I}(P,p_n)$. By induction we conclude. \qed

\vspace{2mm}

\begin{remark}
\em For $d=2$, it can be easily proved that this result is equivalent to \cite[Theorem 3.8]{ur} (see equation (\ref{ok})).
\end{remark}

\begin{remark}
\em Notice that the Betti numbers $\beta_i(H_{P,d})$ do not depend on the characteristic of the base field $\mathbbm{k}$.
\end{remark}

The splitting formula can be simplified if $P$ has a unique maximal element.

\begin{corollary}\label{simpler}
Let $P$ be a poset with a unique maximal element $p_n$ and let $d \geq 2$ be an integer. Then $$\beta_i(H_{P,d})=\beta_i(H_{P \setminus \{p_n\},d})+\beta_i(H_{P,d-1})+\beta_{i-1}(H_{P,d-1}), \text{ for each } i \geq 0.$$
\end{corollary}

\begin{proof} By assumption $P=\langle p_n \rangle$, where $p_n$ is the unique element in $\mathbb{I}(P,p_n)$. The splitting formula follows immediately by equation (\ref{ok}) in the proof of Theorem \ref{cmdip}.  \end{proof} 




Let $G$ be a Cohen-Macaulay bipartite graph. By \cite[Theorem 3.4]{hh2}, there exists a poset $P$ such that $\mathcal{C}_{P,2}=G$. By using Theorem \ref{cmdip}, we are able to give another formulation of \cite[Theorem 3.8]{ur}. 


\begin{corollary}\label{cmbip}
Let $G$ be a Cohen-Macaulay bipartite graph and $P$ be the poset such that $\mathcal{C}_{P,2}=G$. Then $$\beta_i(I(G)^*)=\sum_{\alpha \in \mathbb{I}(P)}\binom{|\mathrm{Max}(\alpha)|}{i}.$$
\end{corollary}

\begin{proof} 
The formula follows immediately by Theorem \ref{cmdip} with $d=2$ and Remark \ref{degenerate}.
\end{proof}



\section{Betti splitting for Cohen-Macaulay simplicial complexes}\label{examples}
In this section we present some interesting examples about Betti splittings of Alexander dual ideals of Cohen Macaulay simplicial complexes. By \cite[Theorem 3]{eagrein} such ideals have a linear resolution.  For more definitions about simplicial complexes, their properties and the Stanley-Reisner correspondence we refer to \cite[Chapter 1]{hibi} and \cite[Chapter 3]{jonsson}.

\begin{definition}
\em An {\em abstract simplicial complex} $\Delta$ on $n$ vertices is a collection of subsets of $\{1,\dots,n\}$, called {\em faces}, such that if $F \in \Delta$, $G \subseteq F$, then $G \in \Delta$.
\end{definition} 

A {\em facet} is a maximal face of $\Delta$ with respect to the inclusion of sets. Denote by $\F(\Delta)$ the collection of facets of $\Delta$. A simplicial complex $\Delta$ is called {\em pure} if all its facets have the same cardinality.
\vskip 2mm
The {\em Alexander dual ideal} $I_{\Delta}^*$ of $\Delta$ is defined by 
\vskip 1.5mm
\begin{center}
$I_{\Delta}^*=(x_{\overline{F}}:F \in \F(\Delta))$, where $\overline{F}:=\{1,\dots,n\} \setminus F$ and $x_{\overline{F}}=\prod_{i \in \overline{F}} x_i$.
\end{center}

We briefly recall the definitions of {\em link} and {\em deletion} of a vertex $i$ of $\Delta$. $$link_{\Delta}(i):=\{F \in \Delta:i \notin F,F \cup \{i\} \in \Delta\}.$$ $$del_{\Delta}(i):=\{F \in \Delta:i \notin F\}.$$

In the following example we show an ideal with a linear resolution that does not admit {\em any} Betti splitting. In this example we also point out that the existence of Betti splittings of $I_{\Delta}$ and $I_{\Delta}^*$ are completely unrelated. 

\begin{ex}\label{dunce}{\em (Dunce hat)}
\em Let $\Delta$ be the two-dimensional simplicial complex in Figure 2. It is a triangulation of the {\em dunce hat} introduced by Zeeman \cite{zee}. 

\begin{figure}[ht!]\centering
\includegraphics[width=60mm]{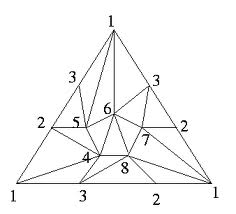}
\caption{A triangulation of the dunce hat.}
\end{figure}

It is contractible but not collapsible (see \cite{jonsson}). Moreover $\Delta$ is Cohen-Macaulay and the graded Betti numbers of $I_{\Delta}^*$ are given by the following linear resolution: $$0 \rightarrow R(-7)^{11} \rightarrow R(-6)^{27} \rightarrow R(-5)^{17} \rightarrow I_{\Delta}^*\text{}.$$ 
In particular $I_{\Delta}^*$ has $17$ minimal generators of degree $5$.

Using CoCoA \cite{cocoa} or Macaulay 2 \cite{mac2}, it can be shown that for {\em every} monomial ideals $J,K$ such that $I_{\Delta}^*=J+K,$ being $G(I_{\Delta}^*)$ is the disjoint union of $G(J)$ and $G(K)$, the decomposition $I_{\Delta}^*=J+K$ is {\em not} a Betti splitting of $I_{\Delta}^*$. This example answers \cite[Question 4.3]{ur}, showing that, in general, there is no relation between the existence of Betti splitting for a square-free monomial ideal and the characteristic dependence of its graded Betti numbers, since in this case the ideal $I_{\Delta}^*$ has characteristic-independent resolution.


Nevertheless the Stanley-Reisner ideal of $\Delta$ admits a Betti splitting $I_{\Delta}=J+K$, where 

$J=(x_2x_6, x_4x_7, x_2x_4x_8, x_3x_4x_5, x_2x_3x_4, x_2x_7x_8,x_3x_4x_6,x_1x_4x_6,x_3x_5x_6)$ and 

\noindent $K=(x_5x_7, x_5x_8,x_1x_2x_5,x_1x_4x_5, x_1x_6x_8, x_1x_6x_7, x_1x_3x_8, x_1x_4x_8, x_1x_3x_7,x_3x_6x_8,x_1x_2x_3,x_3x_7x_8).$

In fact the graded Betti numbers of $I_{\Delta}$ are given by $$0 \rightarrow R(-7)^{11} \rightarrow R(-6)^{50} \rightarrow R(-5)^{86} \rightarrow R(-3)^2 \oplus R(-4)^{65} \rightarrow R(-2)^4 \oplus R(-3)^{17} \rightarrow I_{\Delta},$$ and the graded Betti numbers of $J,K$ and $J \cap K$ are given by
\begin{center}
$0 \rightarrow R(-6)^3 \rightarrow R(-5)^{13} \rightarrow R(-4)^{18} \rightarrow R(-2)^2 \oplus R(-3)^7 \rightarrow J$\\
\vskip 0.3cm
$0 \rightarrow R(-6)^5 \rightarrow R(-5)^{19} \rightarrow R(-3) \oplus R(-4)^{24} \rightarrow R(-2)^2 \oplus R(-3)^{10} \rightarrow K$\\
\vskip 0.3cm
$0 \rightarrow R(-7)^{11} \rightarrow R(-6)^{42} \rightarrow R(-5)^{54} \rightarrow R(-3) \oplus R(-4)^{23} \rightarrow J \cap K$.
\end{center}
\end{ex}


We finally prove that a condition introduced by Nagel and R\"{o}mer in \cite{nagel} ensures the existence of $x_i$-splitting for the Alexander dual of a Cohen-Macaulay simplicial complex.

\begin{definition}{\em (\cite[Definition 3.1]{nagel})}
\em A pure non-empty simplicial complex $\Delta$ on $n$ vertices is called {\em weakly vertex decomposable} (WVD) if there exists a vertex $v$ of $\Delta$ fulfilling exactly one of the following conditions:
\begin{itemize}
\item[(i)] $\mathrm{del}_{\Delta}(v)$ is WVD and $\Delta$ is a cone over $\mathrm{del}_{\Delta}(v)$ with apex $v$;
\item[(ii)] $\mathrm{link}_{\Delta}(v)$ is WVD and $\mathrm{del}_{\Delta}(v)$ is Cohen-Macaulay of the same dimension of $\Delta$.
\end{itemize}
\end{definition}

Notice that Provan and Billera \cite{prbill} use the same name for a different class of complexes. 

Every weakly vertex decomposable simplicial complex is Cohen-Macaulay and every pure vertex decomposable simplicial complex is weakly vertex decomposable, see \cite{nagel}.

Let $\Delta$ be a vertex decomposable complex. By \cite[Theorem 2.8, Corollary 2.11]{moradi}, the ideal $I_{\Delta}^*$ admits $x_i$-splitting, for some $x_i$. Weakly vertex decomposability of $\Delta$  is a sufficient condition for the existence of an $x_i$-splitting of $I_{\Delta}^*$, for a suitable $x_i$, generalizing the result cited above.

\begin{proposition}\label{weakVD}
Let $\Delta$ be a weakly vertex decomposable simplicial complex. Then there exists a vertex $i \in \Delta$ such that $I_{\Delta}^*$ admits $x_i$-splitting.
\end{proposition}

\begin{proof} If $\Delta$ is a cone over $\mathrm{del}_{\Delta}(i)$ for some $i$, then $I_{\Delta}^*$ and $I_{\mathrm{del}_{\Delta}(i)}^*$ have the same graded Betti numbers. Then we may assume that $\Delta$ is not a cone. Let $i$ be the vertex in the definition of weakly vertex decomposable simplicial complex. Since $\mathrm{del}_{\Delta}(i)$ is Cohen-Macaulay, then it is pure and $i$ is a shedding vertex (see for instance \cite{moradi}). By \cite[Lemma 2.2]{moradi}, we have $I_{\Delta}^*=x_iI_{\mathrm{del}_{\Delta}(i)}^*+I_{\mathrm{link}_{\Delta}(i)}^*$. Recall again that $\mathrm{del}_{\Delta}(i)$ is Cohen-Macaulay. Then, by \cite[Theorem 3]{eagrein}, $I_{\mathrm{del}_{\Delta}(i)}^*$ has a linear resolution and the result follows by Proposition \ref{urur} (ii). \end{proof} 

In the following example we show that the converse of the previous proposition does not hold, even in the Cohen-Macaulay case.

\begin{ex}\label{hach}{\em (Hachimori's example)}
\em Let $\Delta$ be the simplicial complex in Figure 3, due to Hachimori \cite{hachi}. It is shellable, hence Cohen-Macaulay., but not weakly vertex decomposable. In fact $\Delta$ is not a cone and $\mathrm{del}_{\Delta}(i)$ is not Cohen-Macaulay for $1 \leq i \leq 7$.

\begin{figure}[ht!]\centering
\includegraphics[width=60mm]{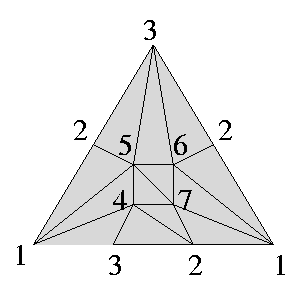}
\caption{Hachimori's example.}
\end{figure} 

The graded Betti numbers of $I_{\Delta}^*$ are given by $$0 \rightarrow R(-6)^8 \rightarrow R(-5)^{20} \rightarrow R(-4)^{13} \rightarrow I_{\Delta}^*.$$ The ideal $I_{\Delta}^*$ admits $x_4$-splitting. In fact $I_{\Delta}^*=x_4J+K$ is a Betti splitting, where 
\begin{center}
$x_4J=(x_3x_4x_5x_7, x_2x_3x_4x_5, x_3x_4x_5x_6,x_1x_2x_3x_4, x_1x_4x_5x_7, x_1x_2x_4x_7, x_1x_4x_6x_7, x_3x_4x_6x_7)$ and\\
\vskip 0.3cm
$K=(x_1x_5x_6x_7, x_1x_3x_5x_6, x_1x_2x_3x_6, x_2x_3x_6x_7, x_2x_5x_6x_7)$. 
\end{center}
The graded Betti numbers of $J,K$ and $J \cap K$ are given by
\begin{center}
$0 \rightarrow R(-6)^4 \rightarrow R(-5)^{11} \rightarrow R(-4)^8 \rightarrow J$\\
\vskip 0.3cm
$0 \rightarrow R(-6) \rightarrow R(-5)^5 \rightarrow R(-4)^5 \rightarrow K$\\
\vskip 0.3cm
$0 \rightarrow R(-6)^3 \rightarrow R(-5)^4 \rightarrow J \cap K$.
\end{center}
\end{ex}

The next definition is important to understand the next example. For further details see \cite{eng} and \cite{jonsson}.

\begin{definition}
\em A $d$-dimensional simplicial complex $\Delta$ is called {\em non evasive} if either
\begin{itemize}
\item[(i)] $\Delta$ is a simplex or
\item[(ii)] $d \geq 1$ and there exists a vertex $v$ of $\Delta$ such that $\mathrm{link}_{\Delta}(v)$ and $\mathrm{del}_{\Delta}(v)$ are both non evasive.
\end{itemize}
\end{definition}


By \cite[Theorem 6.2]{bruno}, every vertex decomposable ball is non evasive. In the next example we present a non evasive ball $\Delta$ such that $I_{\Delta}^*$ does not admits $x_i$-splitting, showing that \cite[Theorem 2.8, Corollary 2.11]{moradi} does not hold for the Alexander dual ideal of a non evasive simplicial complex. 

\begin{ex}\label{rud}{\em (Rudin's ball)}
\em Let $\Delta$ be the {\em Rudin's ball}, a non-shellable triangulation of the tetrahedron with $14$-vertex introduced in \cite{rudin}. Since $\Delta$ is a ball, it is Cohen-Macaulay and $I_{\Delta}^*$ has a $10$-linear resolution: $$0 \rightarrow R(-12)^{30} \rightarrow R(-11)^{70} \rightarrow R(-10)^{41} \rightarrow I_{\Delta}^*.$$ By \cite[Theorem 6.3]{bruno}, $\Delta$ is non evasive. It is not difficult to show that $I_{\Delta}^*$ does not admit $i$-splitting, for $1 \leq x_i \leq 14$. Nevertheless $I_{\Delta}^*=I_{\Delta_1}^*+I_{\Delta_2}^*$ is a Betti splitting of $I_{\Delta}^*$, where $\Delta_1$ and $\Delta_2$ are given by the following facets:
\vskip 4mm
$\mathcal{F}(\Delta_1)\!\!=\!\!\{$\{1,3,7,13\},\{1,3,9,13\},\{1,5,7,11\},\{1,5,9,11\},\{1,7,11,13\},\{1,9,11,13\},\{3,4,7,11\},

\hspace{0.5cm}\{3,4,7,12\},\{3,7,11,14\},\{3,7,12,13\},\{3,9,12,13\},\{4,7,11,12\},\{4,8,11,12\},\{5,6,9,13\},

\hspace{0.5cm}\{5,6,9,14\},\{5,7,11,14\},\{5,9,11,14\},\{5,9,12,13\},\{6,9,13,14\},\{6,10,13,14\},

\hspace{0.5cm}\{7,11,12,13\},\{9,11,13,14\},\{11,12,13,14\}$\}$;
\vskip 2mm
$\mathcal{F}(\Delta_2)\!\!=\!\!\{$\{2,4,8,14\},\{2,4,10,14\},\{2,6,8,12\},\{2,6,10,12\},\{2,8,12,14\},\{2,10,12,14\},\{3,6,10,11\},

\hspace{0.5cm}\{3,6,10,14\},\{3,10,11,14\},\{4,5,8,12\},\{4,5,8,13\},\{4,8,13,14\},\{4,10,13,14\},\{5,8,12,13\},

\hspace{0.5cm}\{6,8,11,12\},\{6,10,11,12\},\{8,12,13,14\},\{10,11,12,14\}$\}$.\\

Both $I_{\Delta_1}^*$ and $I_{\Delta_2}^*$ are computed in $\mathbbm{k}[x_1,\dots,x_{14}]$. The minimal graded free resolutions of $I_{\Delta_1}^*$, $I_{\Delta_2}^*$ and $I_{\Delta_1}^* \cap I_{\Delta_2}^*$ are 
\begin{center}
$0 \rightarrow R(-12)^{13} \rightarrow R(-11)^{35} \rightarrow R(-10)^{23} \rightarrow I_{\Delta_1}^*$\\
\vskip 0.3cm
$0 \rightarrow R(-12)^{10} \rightarrow R(-11)^{27} \rightarrow R(-10)^{18} \rightarrow I_{\Delta_2}^*$\\
\vskip 0.3cm
$0 \rightarrow R(-12)^7 \rightarrow R(-11)^8 \rightarrow I_{\Delta_1}^* \cap I_{\Delta_2}^*\text{ }$.
\end{center}
\end{ex}

\bigskip
\noindent
Davide Bolognini\\
Dipartimento di Matematica\\
Universit{\`a} di Genova\\
Via Dodecaneso 35, 16146 Genova, Italy\\
e-mail: {\tt bolognin@dima.unige.it}
\end{document}